\documentclass{amsart}

\usepackage{amsmath}
\usepackage{amsbsy}
\usepackage{amssymb}
\usepackage{amsthm}
\usepackage{amscd}
\usepackage{mathrsfs}
\usepackage[abbrev]{amsrefs}
\usepackage[all]{xy}

\newcommand{\bbp}{{\mathbb P}}
\newcommand{\bbq}{{\mathbb Q}}
\newcommand{\bbr}{{\mathbb R}}

\newcommand{\bbz}{{\mathbb Z}}

\newcommand{\lrarrow}{\longrightarrow}

\newcommand{\kableadd}%
{Department of Mathematics\\ Cornell University\\
Ithaca NY 14853}

\newcommand{\lan}{\langle}
\newcommand{\ran}{\rangle}

\newtheorem{thm}{Theorem}[section]
\newtheorem{lem}[thm]{Lemma}

\newtheorem{prop}[thm]{Proposition}

\newtheorem{defn}[thm]{Definition}

\newtheorem{rem}[thm]{Remark}        
\newtheorem*{ack*}{Acknowledgment}

\newtheorem*{note*}{Notation and conventions}

\newcommand{\chow}{\operatorname{CH}}
\newcommand{\lisom}{\stackrel{\sim}{\lrarrow}}

\newcommand{\Spec}{\operatorname{Spec}}

\newcommand{\divisor}{\operatorname{div}}

\begin{document}
\pagestyle{headings}

\title[Hasse principle for the Chow groups on quadric fibrations]{On the Hasse principle for the Chow groups of zero-cycles on quadric fibrations}
\author{Kazuki Sato}
\address{Mathematical Institute, Tohoku Univarsity, Sendai, Miyagi, 980-8570, Japan}
\email{sb0m17@math.tohoku.ac.jp}
\begin{abstract}
We give a sufficient condition for the injectivity of the global-to-local map of the relative Chow group of zero-cycles on a quadric fibration of dimension $\leq 3$ defined over a number field.
\end{abstract}
\thanks{}
\maketitle

\renewcommand\baselinestretch{1.1}

%\tableofcontents

\section{Introduction}
Let $k$ be a number field and $\Omega$ the set of its places.
For any variety $X$ over $k$, 
$\chow_0(X)$ denotes the Chow group of zero-cycles on $X$ modulo rational equivalence.
Then we have the global-to-local map
\[\chow_0(X) \lrarrow \prod _{v \in \Omega} \chow_0(X \otimes _k k_v),\]
where for each place $v \in \Omega$, $k_v$ denotes the completion of $k$ at $v$.
If there exists a proper morphism $X \to C$ from $X$ to another variety $C$, we also have the relative version of the global-to-local map
\[\Phi : \chow_0(X/C) \lrarrow \prod _{v \in \Omega} \chow_0(X \otimes _k k_v / C \otimes _k k_v),\]
where $\chow_0(X/C)$ denotes the kernel of the push-forward map $\chow_0(X) \to \chow_0(C)$.
In this paper, we study the injectivity of $\Phi$ for quadric fibrations over curves.

First, let us recall some known results for a surface $X$.
In the case where $X$ is a conic bundle surface over the projective line $\bbp^1_k$, Salberger proved that 
the kernel of $\Phi$ is controlled by the Tate-Shafarevich group of the N\'eron-Severi torus of $X$ 
%there exists the following exact sequence of finite abelian groups
%\[0 \lrarrow {\cyr Sh}^1(k, T) \lrarrow A_0(X) \stackrel{\Phi}{\lrarrow} \bigoplus _{v \in \Omega} A_0(X \otimes _k k_v) \lrarrow \Hom ( H^1(k, \hat{T}), \bbq/ \bbz),\]
%where $A_0(X)$ is the kernel of the degree homomorphism $\deg : \chow_0(X) \to \bbz$ 
(see \cite{sal} for the details).
%This result was generalized by Colliot-Th\'el\`ene to the case where the genus of the base curve is arbitrary \cite{ct}.
%He showed that if $X$ is a conic bundle surface over a smooth projective curve $C$ of arbitrary genus over a totally imaginary number field $k$, then the image of $\Phi$ is just the kernel of the paring with the Brauer group of $X$.
%\[ \chow_0 (X/C) \lrarrow \bigoplus _v \chow_0(X_v/C_v) \lrarrow \Hom( \operatorname{Br}(X)/ \operatorname{Br}(C), \bbq / \bbz ),\]
%where $\chow_0(X/C)$ denotes the kernel of the map $\chow_0(X) \to \chow_0(C)$.  
%Afterward Frossard \cite{fro2} extended above exact sequence in the case that $X$ is a Severi-Brauer fibration over a curve (i.e. a variety with a proper dominant morphism into a curve $C$ whose generic fiber is a Severi-Brauer variety).

In the case where $X$ is a quadric fibration of dimension $\geq 4$, few results are known.
Parimala and Suresh proved that if $X \to C$ is a quadratic fibration over a smooth projective curve $C$ whose generic fiber is defined by a Pfister neighbour of rank $\geq 5$, then the global-to-local map restricted to real places
 \[ \Phi _{\mathrm{real}} : \chow_0(X/C) \lrarrow \bigoplus _{v: \mathrm{real \ place}} \chow_0(X \otimes _k k_v/C \otimes _k k_v)\]
 is injective \cite{ps}.
By using this injectivity, they deduced a finiteness result of the torsion subgroup of the Chow group $\chow_0(X)$ of zero-cycles on $X$.

When $X \to C$ is a quadric fibration of $\dim \leq 3$, not only that the map $\Phi_{\mathrm{real}}$ is not injective in general,
but also the map $\Phi$ is not injective \cite{suresh}.
However, the map $\Phi$ can be injective.
If the generic fiber of $X \to C$ is defined by a quadratic form over a base field $k$, the map $\Phi$ is injective (Theorem \ref{mt}).
The above condition does not imply the injectivity of $\Phi_{\mathrm{real}}$, and we give an example of this (Propositon \ref{mt2}).
Note that we don't assume that quadric fibrations are admissible (for the definition of admissibility, see \cite{ps}).

%We give a counterexample (Proposition \ref{mt2}).

%\On the other hand, the map $\Phi$ can b 
%By above results はおかしいので何か修正する
%By above results, the following question arises; 
%are the maps $\Phi$ and $\Phi_{\mathrm{real}}$ injective if $\dim X \leq 3$?
%On the injectivity of $\Phi$, it is known that in general $\Phi$ is not injective \cite{suresh}.
%However, we give a sufficient condition for the injectivity of $\Phi$ (Theorem \ref{mt}).
%In addition, we show that $\Phi_{\mathrm{real}}$ is not injective under the above condition even though $\Phi$ is injective and the generic fiber is \defined by a Pfister form (Proposition \ref{mt2}).

\begin{note*}\normalfont
In section 2, $k$ denotes a field of characteristic different from $2$.
In section 3, $k$ denotes a number field (i.e. a finite extension field of $\bbq$).
For a variety $X$, $| X |$ denotes the set of closed points on $X$.
%For a variety $X$ over $k$ and a integer $i \geq 0$, $X_{(i)}$ (resp. $X^{(i)}$ ) denotes the set of points of dimension $i$ (resp. codimension $i$) in $X$.
We denote by $\chow_i(X)$ the Chow group of cycles of dimension $i$ on $X$ modulo rational equivalence \cite{fulton}.
For a geometrically integral variety $X$ over $k$, $k(X)$ denotes the function field of $X$.
For any extension $L/k$ of fields, $L(X)$ denotes the function field of $X \otimes _k L$.
If $x$ is a point in $X$, $k(x)$ denotes the residue field at $x$.
\end{note*}

\begin{ack*}\normalfont
I would like to thank Professor M. Hanamura for helpful comments and warm encouragement.
\end{ack*}

\section{Definition of the map $\delta$}
%Preliminaryは複数形で使う？
%Notationと基本的な定義のセクションは分けたほうが見やすいかも。
By a \textit{quadratic space} over $k$, we mean a nonsingular quadratic form over $k$.
We denote by $W(k)$ the Witt group of quadratic spaces over $k$ and by $Ik$ the fundamental ideal of $W(k)$ consisting of classes of even rank quadratic spaces.
We represent quadratic spaces over $k$ by diagonal matrices $\lan a_1, \dots , a_n \ran ( a_i \in k^{*})$ with respect to the choice of an orthogonal basis.
By an $n$-fold \textit{Pfister form} over $k$, we mean a quadratic space of the type $\lan 1,a_1 \ran \otimes \lan 1, a_2 \ran \otimes \dots \otimes \lan 1, a_n \ran$.
The set of nonzero values of a Pfister form is a subgroup of the multiplicative group $k^*$ of $k$ \cite[Theorem 1.8, p. 319]{lam}.
By a \textit{Pfister neighbor} of an $n$-fold Pfister form $q$, we mean a quadratic space of rank at least $2^{n-1} +1$ which is a subform of $q$ \cite[Example 4.1]{kn2}.

For any quadratic space $q$ over $k$, let $N_q(k)$ be the subgroup of $k^*$ generated by norms from finite extensions $E$ of $k$ such that $q$ is isotropic over $E$.
If $q$ is isotropic, then clearly $N_q(k) = k^*$.
For any $a \in k^*$, $q$ is isotropic if and only if $\lan a \ran \otimes q$ is isotropic.
Therefore $N_q(k) = N_{\lan a \ran \otimes q}(k)$.
By Knebusch's norm principle, $N_q(k)$ is generated by elements of the form $xy$, with $x, y \in k^*$ which are values of $q$ over $k$ \cite[Lemme 2.2]{cts}. 
In particular, if a quadratic form $q$ is of the form $q = \lan 1, a \ran \otimes \lan 1, b \ran$, then $N_q(k)$ is equal to the group $\operatorname{Nrd}_{D / k}( D^{*})$ of reduced norms of the quaternion algebra $D = (-a, -b )_k$.
Suppose that $q'$ is a Pfister neighbor of a Pfister form $q$.
Then, for any extension $E/k$, $q'$ is isotropic over $E$ if and only if $q$ is isotropic over $E$ \cite[Example 4.1]{kn2}.
So we have $N_q(k) = N_{q'}(k)$.

The following lemma is elementary and well-known, but the proof does not seem to be written explicitly in the literature.
 
\begin{lem}
Let $q$ be a Pfister form over $k$.
Then
\[ N_q(k) = \{ x \in k^* \mid q \otimes \lan 1, -x \ran \ \textrm{is isotropic}\}. \]
In particular, $x$ belongs to $N_q(k)$ if and only if $q \otimes \lan 1, -x \ran = 0$ in $W(k)$.
\end{lem}

\begin{proof}
Since $q$ is a Pfister form, $N_q(k)$ is the set of non-zero values of $q$.
If $q$ is isotropic, the assertion is clear.
We suppose that $q$ is anisotropic.
If $q \otimes \lan 1, -x \ran$ is isotropic, then there exist two vectors $v_1, v_2$ in the underlying vector space of $q$, not both zero, such that $q(v_1) -xq(v_2) = 0$.
%%%%%%%%%%%%%%%%%%%%%%%%%%%%%%%%%%%%%%%%%%%%%%%%%%%%%%%%%%%%%%%%%%%%%%%%%%%%%%%not both zeroのより適切な表現を調べる%%%%%%%%%%%%%%%%%%
Therefore
\[ x= q(v_1)/q(v_2) \in N_q(k).\]
The other implication follows from the fact that a Pfister form represents $1$.

The last assertion results from the basic fact on Pfister forms \cite[Theorem 1.7, p. 319]{lam}.
\end{proof}

\begin{defn}\normalfont
Let $C$ be a smooth projective geometrically integral curve over $k$.
A \textit{quadric fibration} ($X, \pi$) over $C$ is a geometrically integral variety $X$ over $k$, together with a proper flat $k$-morphism $\pi : X \to C$ such that each point $P$ of $C$ has an affine neighborhood $\Spec A(P)$, with $X \times _C \Spec A(P)$ isomorphic to a quadric in $\bbp^n_{A(P)}$ and such that the generic fiber of $\pi$ is a smooth quadric.
\end{defn}

Given a quadratic space $q$ over the function field $k(C)$ of $C$ of rank $\geq 3$,
we can easily construct a quadric fibration $\pi :X \to C$, 
whose generic fiber is given by the quadratic space $q$.
It is not unique, but two quadric fibrations having the same generic fiber are birational over $C$.

%up to birationalという表現を改める

Let $\pi : X \to C$ be a quadric fibration with the generic fiber given by a quadratic space $q$, and $\chow_0(X/C)$ denote the kernel of the map
\[ \pi _* : \chow_0(X) \lrarrow \chow_0(C) .\]
We have the following commutative diagram with exact rows (see \cite{cts})
\[ \displaystyle
\begin{CD}
@. \displaystyle\bigoplus_{x \in |X_{\eta}| } k(x)^{*} @>>> \displaystyle\bigoplus_{P \in |C|} \chow_0(X_P) @>>> \chow_0(X) @>>> 0 \\
@. @VV\oplus N_{k(x) / k(C)}V @VV\oplus\deg_{X_P /k(P)}V @VV\pi_{*}V \\
0 @>>> k(C)^{*}/k^{*} @>>> \displaystyle\bigoplus_{P \in |C|} \bbz @>>> \chow_0(C) @>>> 0, 
\end{CD}
\]
where $X_\eta$ is the generic fiber of $\pi: X \to C$ and $\deg_{X_P /k(P)}: \chow_0(X_P) \to \bbz$ is the degree map.
Since the map $ \bigoplus _{x \in |X_{\eta}|} k(x)^* \to k(C)^*$ is induced by norms, 
the image is precisely $N_q(k(C))$.
By the snake lemma and the fact that $A_0(X_P) = 0$ (\cite{sw}),
we have an exact sequence
\[ 0 \lrarrow \chow_0(X/C) \stackrel{\delta}{\lrarrow} k(C)^* / k^* N_q(k(C)) \lrarrow \bigoplus_{P \in C^{(1)}} \bbz/\deg_{X_P /k(P)}(\chow_0(X_P)). \]

\begin{rem}\normalfont
We denote by $k(C)_{\mathrm{dn}}^*(q)$ the subgroup of $k(C)^{*}$ consisting of functions, which, at each closed point $P \in C$,
can be written as a product of a unit at $P$ and an element of $N_q(k(C))$.  
Colliot-Th\'el\`ene and Skorobogatov \cite{cts} proved that the above homomorphism $\delta$ defines an isomorphism
\[ \delta : \chow_0(X/C) \lisom k(C)_{\mathrm{dn}}^*(q) / k^* N_q(k(C)) \]
for an \textit{admissible} quadric fibration $\pi: X \to C$.
However, we don't have to assume the admissibility for our main results.

\end{rem}

\section{Injectivity of the global-to-local map}
%主定理で使われる定義、記号をきちんと定義する。
%証明の議論を精密に、厳密に説明する。
Let $k$ be a number field and $\Omega$ be the set of places of $k$.
For any $v \in \Omega$, $k_v$ denotes the completion of $k$ at $v$.

In \cite[Theorem 5.4]{ps}, Parimala and Suresh proved that if $X \to C$ is an admissible quadric fibration whose generic fiber is given by a Pfister neighbor over $k(C)$ of rank $\geq 5$, then the map
\[ \Phi_{\mathrm{real}} : \chow_0(X/C) \lrarrow \bigoplus _{v:   \mathrm{real \ places}} \chow_0(X\otimes _k k_v/C\otimes _k k_v) \]
is injective, where $v$ runs over all real places of $k$.
In the case where the rank of the quadratic form defining the generic fiber is less than $5$, we give a sufficient condition for the injectivity of the global-to-local map $\Phi$.

\begin{thm}\label{mt}
Let $\pi: X \to C$ be a quadric fibration over a number field $k$.
Assume that $\dim X = 2$ or $3$, and the generic fiber of $\pi$ is isomorphic to a quadric defined over $k$ $($i.e. there exists a quadric $Q \subset \bbp_k^{N}$ such that the generic fiber is isomorphic to $Q \otimes _k k(C)$ over $k(C)$
$)$.
Then, the natural map
\[ \Phi : \chow_0(X/C) \lrarrow \bigoplus _{v \in \Omega} \chow_0(X\otimes _k k_v/C\otimes_k k_v) \]
is injective.
\end{thm}

\begin{proof}
Let $q$ be a quadratic form defining the generic fiber of the quadric fibration $\pi: X \to C$.
In order to prove the theorem, it is sufficient to show that the natural map
\[ k(C)^* / k^* N_q(k(C)) \lrarrow \prod_{v \in \Omega} k_v(C)^* / k_v^* N_q(k_v(C)) \]
is injective.

We may assume that $q= \lan 1, a, b, abd\ran, a, b, d \in k^*$.
Put $L := k( \sqrt{d})$.
Note that $q$ is isometric to $\lan 1, a \ran \otimes \lan 1, b \ran$ over $L(C)$.
Let $f \in k(C)^*$ such that $f \in k_v^*N_q(k_v(C))$ for all places $v$ of $k$.
For any real place $w$ of $L$, denote by $w'$ the restriction of $w$ to $k$.
Since $f \in k_{w'}^*N_q(k_{w'}(C))$, there exists $\mu_{w'} \in k^*_{w'}$ such that $\mu_{w'}f \in N_q(k_{w'}(C))$.
We can choose $\mu \in k^*$ such that the sign of $\mu$ is the same as that of $\mu_{w'}$ for each real place $w$ of $L$.
Thus we have $\mu f \in N_q(k_{w'}(C)) \subset N_q(L_w(C))$.
Therefore $q \otimes \lan 1, -\mu f \ran$ is hyperbolic over $L_w(C)$ for each real place $w$ of $L$.
For a complex place $w$ of $L$, it is clear that $q \otimes \lan 1, -\mu f \ran$ is hyperbolic over $L_w(C)$.
Further, for a finite place $w$ of $L$, we have $f \in k_v^*N_q(k_v(C))$, where $v$ is the place of $k$ below $w$. 
Since $k_v^* \subset N_q(k_v(C))$,
\[ \mu f \in k_v^*N_q(k_v(C)) = N_q(k_v(C)) \subset N_q(L_w(C)).\]
Therefore $q \otimes \lan 1, -\mu f \ran $ is hyperbolic over $L_w(C)$ for all places $w$ of $L$.
By \cite[Theorem 4]{aej}, the natural map
\[ I^3L(C) \lrarrow \prod _w I^3L_w(C) \]
is injective, where $w$ runs over all places of $L$.
Hence $q \otimes \lan 1, -\mu f \ran $ is hyperbolic over $L(C)$.
By \cite[Proposition 2.3]{cts}, we have
\[ \mu f \in N_q(L(C)) \cap k(C)^* = N_q(k(C)) .\]
This proves the required injectivity.

The image of the global-to-local map $\Phi$ lies in the direct sum $\bigoplus _v \chow_0(X\otimes _k k_v/C\otimes_k k_v)$.
Indeed, a quadratic form of rank $4$ defined over a number field $k$ is isotropic over $k_v$ for all but finitely many places $v$ of $k$.
Therefore for all but finitely many $v$, $\chow_0(X\otimes _k k_v /C\otimes _k k_v) =0 $.
\end{proof}

\begin{rem}\normalfont
Without the assumption that the generic fiber is defined over $k$, the natural map
\[ \Phi : \chow_0(X/C) \lrarrow \prod _{v \in \Omega} \chow_0(X\otimes _k k_v/C\otimes _k k_v) \]
is not injective \cite{suresh}.

\end{rem}
%%%%%%%%%%%%5propoの説明を入れる%%%%%%%%%%%%%%%%%%%55
Finally, we consider the restricted global-to-local map $\Phi _{\mathrm{real}}$.
Parimala and Suresh's result \cite[Theorem 5.4]{ps}, which is for quadratic forms of rank at least $5$, does not hold for forms of smaller rank.
We give the following example, which is an variation of \cite[Proposition 6.1]{ps}.
%When $\dim X = 3$, there exists a counterexample to \cite[Theorem 5.4]{ps}.

\begin{prop}\label{mt2}
Let $C$ be the elliptic curve over $\bbq$ defined by
\[ y^2 = -x(x+2)(x+3).\]
Assume that the generic fiber of a quadric fibration $\pi : X \to C$ is isomorphic to the quadric defined by the quadratic form $q= \lan 1, -2, 3, -6 \ran$.
Then the natural map
\[ \Phi_{\mathrm{real}} : \chow_0(X/C) \lrarrow \chow_0(X \otimes _\bbq \bbr/ C \otimes _\bbq \bbr) \]
is not injective.
\end{prop}

\begin{proof}
Since $q$ is isotropic over $\bbr$, $N_q(\bbr(C)) = \bbr(C)^*$.
So we have $\chow_0(X \otimes _\bbq \bbr/ C \otimes _\bbq \bbr) =0$.

Since $\divisor_C(x) = 2D$, for some divisor $D$ on $C$, $x \in \bbq(C)^*/\bbq^* N_q(\bbq(C))$ is contained in $\operatorname{Im}\delta$. 
On the other hand, $q$ is isometric to $\lan 1, 1, 3, 3 \ran$ over $\bbq_3$.
Thus $x \notin \bbq_3^* N_q(\bbq_3(C))$ by \cite[Proposition 6.1]{ps}.
Therefore we have $\chow_0(X/C) \neq 0$.
\end{proof}

\begin{rem}\normalfont
Note that in the above case the map
\[ \Phi : \chow_0(X/C) \lrarrow \bigoplus _{v \in \Omega} \chow_0(X\otimes _k k_v/C\otimes _k k_v) \]
is injective by Theorem\ref{mt}.
The summands of the right hand side vanish except for $\chow_0(X \otimes _\bbq \bbq_2/C \otimes _\bbq \bbq_2)$ and $\chow_0( X \otimes _\bbq \bbq_3/C \otimes _\bbq \bbq_3)$, since the quadratic form $\lan 1, -2, 3, -6 \ran$ is isotropic over $\bbr$ and over $\bbq_p$ for all primes $p$ except 2 and 3.
\end{rem}


\begin{bibdiv}
\begin{biblist}

\bib{aej}{article}{
   author={Arason, J{\'o}n Kr.},
   author={Elman, Richard},
   author={Jacob, Bill},
   title={Fields of cohomological $2$-dimension three},
   journal={Math. Ann.},
   volume={274},
   date={1986},
   number={4},
   pages={649--657},
 
}



%\bib{ct}{article}{
  % author={Colliot-Th{\'e}l{\`e}ne, Jean-Louis},
   %title={Principe local-global pour les z\'ero-cycles sur les surfaces
   %r\'egl\'ees},
%    language={French},
%    note={With an appendix by E. Frossard and V. Suresh},
%    journal={J. Amer. Math. Soc.},
%    volume={13},
%    date={2000},
%    number={1},
%    pages={101--127},

% }



\bib{cts}{article}{
   author={Colliot-Th{\'e}l{\`e}ne, Jean-Louis},
   author={Skorobogatov, Alexei N.},
   title={Groupe de Chow des z\'ero-cycles sur les fibr\'es en quadriques},
   
   journal={$K$-Theory},
   volume={7},
   date={1993},
   number={5},
   pages={477--500},
 
}



\bib{fulton}{book}{
   author={Fulton, William},
   title={Intersection theory},
   series={Ergebnisse der Mathematik und ihrer Grenzgebiete. 3. Folge. A
   Series of Modern Surveys in Mathematics [Results in Mathematics and
   Related Areas. 3rd Series. A Series of Modern Surveys in Mathematics]},
   volume={2},
   edition={2},
   publisher={Springer-Verlag},
   place={Berlin},
   date={1998},
   pages={xiv+470},
   
}




\bib{kn2}{article}{
   author={Knebusch, Manfred},
   title={Generic splitting of quadratic forms. I},
   journal={Proc. London Math. Soc. (3)},
   volume={33},
   date={1976},
   number={1},
   pages={65--93},
  
}





\bib{lam}{book}{
   author={Lam, T. Y.},
   title={Introduction to quadratic forms over fields},
   series={Graduate Studies in Mathematics},
   volume={67},
   publisher={American Mathematical Society},
   place={Providence, RI},
   date={2005},
   pages={xxii+550},
  
}













\bib{ps}{article}{
   author={Parimala, R.},
   author={Suresh, V.},
   title={Zero-cycles on quadric fibrations: finiteness theorems and the
   cycle map},
   journal={Invent. Math.},
   volume={122},
   date={1995},
   number={1},
   pages={83--117},
   
}
	
	

\bib{sal}{article}{
   author={Salberger, P.},
   title={Zero-cycles on rational surfaces over number fields},
   journal={Invent. Math.},
   volume={91},
   date={1988},
   number={3},
   pages={505--524},
}

	

\bib{suresh}{article}{
   author={Suresh, V.},
   title={Zero cycles on conic fibrations and a conjecture of Bloch},
   journal={$K$-Theory},
   volume={10},
   date={1996},
   number={6},
   pages={597--610},
}







\bib{sw}{article}{
   author={Swan, Richard G.},
   title={Zero cycles on quadric hypersurfaces},
   journal={Proc. Amer. Math. Soc.},
   volume={107},
   date={1989},
   number={1},
   pages={43--46},
   
}






	



\end{biblist}
\end{bibdiv}
\end{document}